\documentclass[11pt]{amsart}

\usepackage{amsfonts,amsmath,latexsym,amssymb,verbatim,amsbsy,amsthm}

\usepackage[top=1in, bottom=1in, left=1in, right=1in]{geometry}

\usepackage[dvipsnames]{xcolor}

\usepackage[colorlinks=true, pdfstartview=FitV, linkcolor=RoyalBlue,citecolor=ForestGreen, urlcolor=blue]{hyperref}

\theoremstyle{plain}
\newtheorem{THEOREM}{Theorem}[section]

\newtheorem{LEMMA}[THEOREM]{Lemma}

\newtheorem{lemma}[THEOREM]{Lemma}

\theoremstyle{definition}

\theoremstyle{remark}
\newtheorem{REMARK}[THEOREM]{Remark}

\def \a {\alpha}

\def \d {\delta}
\def \e {\varepsilon}

\def \l {\lambda}

\def \t {\tau}

\def \cH {\mathcal{H}}

\def \cL {\mathcal{L}}

\newcommand{\Z}{\ensuremath{\mathbb{Z}}}   
\newcommand{\R}{\ensuremath{\mathbb{R}}}   
\newcommand{\T}{\ensuremath{\mathbb{T}}}   

\def \p {\partial}

\def \ss {\subset}

\renewcommand{\geq}{\geqslant}
\renewcommand{\ge}{\geqslant}
\renewcommand{\leq}{\leqslant}
\renewcommand{\le}{\leqslant}

\DeclareMathOperator{\supp}{supp} %
\DeclareMathOperator{\Lip}{Lip} %
\def \dx  {\, \mbox{d}x}
\def \dt  {\, \mbox{d}t}

\def \dl  {\, \mbox{d}l}
\def \dy  {\, \mbox{d}y}
\def \dz  {\, \mbox{d}z}
\def \dr  {\, \mbox{d}r}
\def \ds  {\, \mbox{d}s}

\def \dtau  {\, \mbox{d}\tau}

\def \dd  {\mbox{d}}

\def \brho {\bar{\rho}}

\begin{document}

\title{On the Structure of Limiting Flocks in Hydrodynamic Euler Alignment Models}

\author{Trevor M. Leslie}
\address{Department of Mathematics, University of Wisconsin, Madison}
\email{tleslie2@wisc.edu}

\author{Roman Shvydkoy}
\address{Department of Mathematics, Statistics, and Computer Science, 
	University of Illinois, Chicago}
\email{shvydkoy@uic.edu}

\date{\today}



\thanks{\textbf{Acknowledgment.} RS was supported in part by NSF grants DMS-1515705, DMS-1813351, and the Simons Foundation.  TL was supported in part by NSF grants DMS-1147523 (PI: Andreas Seeger) and DMS 1515705 (PI: Roman Shvydkoy).}

\begin{abstract}
The goal of this note is to study limiting behavior of a self-organized continuous flock evolving according to the 1D hydrodynamic Euler Alignment model. We provide a series of quantitative estimates that show how far the density of the limiting flock is from a uniform distribution. The key quantity that controls density distortion is the entropy $\cH = \int \rho \log \rho \dx$, and the measure of deviation from uniformity is given by a well-known conserved quantity $e = u' + \cL_\psi \rho$, where $u$ is velocity and $\cL_\psi $ is the communication operator with kernel $\psi$. The cases of Lipschitz, singular geometric, and topological kernels are covered in the study. 
\end{abstract}


\maketitle

\section{Introduction}

In this note we continue the study of the long-time behavior of solutions to the following Euler-alignment model on the torus $\T = [-\pi, \pi]$:
\begin{equation}
\label{e:maind}
\rho_t + (\rho u)' = 0,
\end{equation}
\begin{equation}
\label{e:mainv}
u_t + u u' = \int_{\T} \psi(x,y)(u(y)-u(x))\rho(y)\dy = \cL_\psi(\rho u) - u\cL_\psi(\rho),
\end{equation}
where 
\[
\cL_\psi f := \int_\T \psi(x,y)(f(y) - f(x))\dy,
\]
and $\psi:\T^2 \to [0,\infty)$ is a communication  kernel. (Here and below, we use primes to denote spatial derivatives: $f' = \p_x f$.)  This model represents a one-dimensional hydrodynamic analogue of the Cucker-Smale agent based dynamical system \cite{CS2007a,CS2007b}, and found application in a wide variety of subjects, see \cite{MT2014,VZ2012,CCP2017} for recent surveys. The system \eqref{e:maind} -- \eqref{e:mainv} is designed to describe the mechanism of alignment of congregations of agents governed by laws of self-organization with communication encoded into the kernel $\psi$. The  long time behavior is thus characterized by convergence to a flocking state, by which we mean alignment to a constant velocity $u \to \bar{u}$, and stabilization of density to a traveling wave 
\begin{equation}\label{e:flockwave}
\rho(x,t)\to \rho_\infty(x - t\bar{u}).
\end{equation}
Such a result was proved under the strong global communication condition $\inf \psi >0$ on the torus, see \cite{ST2,ST3}; on the open space, exponential alignment and bounded support of density (weak flocking) was shown under a weaker ``fat tail" condition $\int^\infty \psi(r) \dr =\infty$ in a variety of settings, \cite{CCTT2016,TT2014,HT2008}. The case of local kernels, by which we mean purely local protocols, $\supp \psi \ss \{|x-y|\leq R\}$, remains largely open with the exception of a new class of topological kernels introduced in \cite{ST-topo}, and the case of strong communication relative to other initial parameters of the data, \cite{MPT2018}.

The multitude of flocking states \eqref{e:flockwave} demonstrates that the Euler alignment system \eqref{e:maind} -- \eqref{e:mainv} supports a variety of self-organization outcomes. However it is hard to predict what that outcome $\rho_\infty$ will be from initial conditions. Note that $\bar{u}$, on the other hand, is uniquely determined by the ratio of conserved momentum over mass. In this article we propose to study this question with a less ambitious goal: determine how far the limiting flock  $\rho_\infty$ deviates from the uniform distribution $\bar{\rho} = \frac{1}{2 \pi} M_0$, where $M_0$ is the total mass. In fact, we consider a more general case when the convergence \eqref{e:flockwave} is unknown. As a measure of ``disorder" of the flock we consider the long time limit
\begin{equation}\label{e:disorder}
\limsup_{t \to \infty} \| \rho(\cdot,t) - \bar{\rho}\|_{L^1(\T)}.
\end{equation}
Let us recall that in 1D the Euler alignment system possesses an extra conserved quantity (see \cite{TT2014,CCTT2016,ST-topo})
\[
e_t + (u e)' = 0, \quad e = u' + \cL_\psi \rho,
\]
provided $\psi$ is either of convolution type, $\psi(x,y) = \psi(x-y)$, or topological type as defined below. The physical nature of this quantity has remained elusive, but we will find that it is directly implicated in quantifying disorder of the flock similar to the topological entropy. More precisely, we consider the $e$-quantity per mass of the flock:  $q = \frac{e}{\rho}$. Note that $q$ is transported:
\begin{equation}\label{e:transport}
	q_t + u q' = 0.
\end{equation} 
This allows to trace information at any time $t$ back to the initial datum, in particular, $\|q(\cdot,t)\|_\infty = \|q_0\|_\infty$. The thrust of our main results is to show that the latter is the parameter that controls deviation from the uniform flock expressed by the limit \eqref{e:disorder}. 

Let us set some assumptions.  We distinguish two classes of kernels:
\begin{itemize}
\item Lipschitz  convolution type kernels $\psi \in \Lip(\T)$ with local communication 
\begin{equation}\label{e:Lipker}
	 \psi(x-y)\ge \l \chi_{R_0}(|x-y|), \quad R_0, \l>0;
\end{equation}
\item Symmetric topological kernels, as introduced in \cite{ST-topo}\footnote{We cite here the first draft of the manuscript \cite{ST-topo}; in later versions the authors specialize to the case where $\tau=n$, where $n$ is the dimension of the space.}: here $\t\geq 0$, $0<\a<2$,
\begin{equation}\label{e:topoker}
	\psi(x,y,t) = \frac{h(x-y)}{|x-y|^{1+\a-\tau} \dd(x,y,t)^\tau}, 
	\quad 
	h(x-y) \geq \l \chi_{R_0}(|x-y|),
\end{equation}
where 
\[ 
\dd(x,y,t) = \left| \int_x^y \rho(z,t)\dz \right|. 
\]	
\end{itemize}
The well-posedness theory for the smooth case was developed in \cite{CCTT2016,TT2014}, and for the singular case in \cite{DKRT2018,ST1,ST2,ST3} (geometric kernels, $\tau=0$) and \cite{ST-topo} (topological kernels, $\t>0$). In the smooth kernel setting, the threshold condition $u_0' + \psi * \rho_0\ge 0$ guarantees global existence,
while in the singular case, one has existence for any data if $\tau\leq \a$ and for small data if $\tau >\a$.

We now state our main results. 
\begin{THEOREM}
\label{t:exp}
Let $(\rho, u)$ be a smooth solution to the system \eqref{e:maind}--\eqref{e:mainv}, with kernel given by either \eqref{e:Lipker} or \eqref{e:topoker}.  If  $e_0 = 0$, then 
\begin{equation}
\label{e:exprho}
\|\rho(t) - \brho\|_{L^1} \le c_1(\|\rho_0\|_{L^2}) e^{-\l c_2(R_0,M_0,\a,\t, \|\rho_0\|_{L^\infty}) t},
\end{equation}
where $c_2$ depends only on $R_0$ and $M_0$ in the Lipschitz case, and $c_2$ need not depend on $\|\rho_0\|_{L^\infty}$ if $\t\le \a + 1$ in the topological case.
\end{THEOREM}
We note that this result complements the one obtained in \cite{ST-topo} for topological case. Namely, if $e_0 = 0$, then in the $L^\infty$ metric one has a slower algebraic relaxation towards the uniform state:
\[
\|\rho(t) - \brho\|_{L^\infty} \lesssim \frac{1}{\sqrt{1+t}},
\]
as $t \to \infty$. 
\begin{THEOREM}
	\label{t:Lip}
	Let $(\rho, u)$ be a smooth solution to the system \eqref{e:maind}--\eqref{e:mainv}, with a Lipschitz kernel $\psi$ satisfying \eqref{e:Lipker}. Provided $\|q_0\|_{L^\infty} < \|\psi\|_{L^1}$, one has
	\begin{equation}
	\label{e:Lip}
	\limsup_{t \to \infty} \| \rho(\cdot,t) - \bar{\rho}\|_{L^1}  \leq \frac{ M_0  \|q_0\|_{L^\infty} \|\psi\|_{L^\infty}}{ \l  c(R_0) (\|\psi\|_{L^1} - \|q_0\|_{L^\infty}) },
	\end{equation}
	where $c(R_0)$ is a constant depending only on $R_0$. 
\end{THEOREM}
Let us note that the dependence on $\|q_0\|_{L^\infty}$ is linear for small values. At the same time, the bound is inversely proportional to the strength $\l$, which shows the stabilizing effect of communication on the structure of the flock.

\begin{THEOREM}
\label{t:topo}
Let $(\rho, u)$ be a smooth solution to \eqref{e:maind}--\eqref{e:mainv}, with topological kernel $\psi$. One has the following bounds for any initial data 
\[
\limsup_{t \to \infty} \| \rho(\cdot,t) - \bar{\rho}\|_{L^1}  \leq
\left\{
\begin{array}{lcl}
c \l^{-1} M_0^{1+\t}  \|q_0\|_{L^\infty} [(\a - \t) \l^{-1} M_0^\t \|q_0\|_{L^\infty} + R_0^{\t-\a}]^{\frac{1}{\a-\t}}, & & 0\leq \t<\a \\
c \l^{-1} M_0^{1+\t}  \|q_0\|_{L^\infty}  \exp\left( 1 + \l^{-1} M_0^\t \|q_0\|_{L^\infty} \right),
& & \t = \a,
\end{array}\right.
\]
where $c = c(R_0,\a,\t)$. And one has the following two bounds 
\[
\limsup_{t \to \infty} \| \rho(\cdot,t) - \bar{\rho}\|_{L^1}  \leq
\left\{
\begin{array}{lcl}
c \l^{-1} M_0^{1+\t}  \|q_0\|_{L^\infty}  [(\a - \t) \l^{-1} M_0^\t \|q_0\|_{L^\infty} + R_0^{\t-\a}]^{\frac{1}{\a-\t}} & & \a < \t < 1+ \a \\
c \l^{-1} M_0^{1+\t} \|q_0\|_{L^\infty} [(\a - \t) \l^{-1} M_0^\t \|q_0\|_{L^\infty} + R_0^{\t-\a}]^{\a-\t} & & \t \geq 1+\a,
\end{array}\right.
\]
under the smallness requirement
\[
\|q_0\|_{L^\infty} < \frac{\l M_0^{-\t} R_0^{\t-\a}}{\t-\a}.
\]
\end{THEOREM}
Let us note that for small $q_0$ all the bounds are essentially linear in  $\|q_0\|_{L^\infty}$. Other factors minimize distortions of the flock as well, such as strength of communication $\l >>1$, or small mass. 

\begin{REMARK}
We actually only need lower bounds on $q_0$ in the smallness conditions of the Theorems above.  For example, in \eqref{e:Lip}, the quantity $\|\psi\|_{L^1} - \|q_0\|_{L^\infty}$ can be replaced with $\|\psi\|_{L^1} + \inf q_0$; one can see this by following the estimates on the density amplitude in Section \ref{s:densamp}.  This results in the slightly relaxed smallness requirement $\inf q_0>-\|\psi\|_{L^1}$.  A similar remark applies to the statement for topological kernels.  However, since our primary interest is in the case where $\|q_0\|_{L^\infty}$ is small, we state our results entirely in terms of $\|q_0\|_{L^\infty}$, for simplicity.  

Note also that in the case of Lipschitz kernels, $\inf q_0\ge -\|\psi\|_{L^1}$ (with non-strict inequality) is equivalent to the critical threshold condition (c.f. \cite{CCTT2016}, \cite{ST1}) that dictates whether the solution remains globally smooth or blows up in finite time.  So the precise version of the smallness condition we need in order to control the density is almost guaranteed already by the fact that we are working with globally smooth solutions.  The situation is similar for topological kernels in the case $\t>\a$: the smallness condition is precisely the one from the existence theory in \cite{ST-topo}.  Unlike the case of Lipschitz kernels, however, it is not known how sharp this condition is for the existence theory for topological kernels. 
\end{REMARK}

\pagebreak

\section{Proofs of the results}
Let us recall some preliminary facts before we proceed to the main proofs. 

\subsection{The Csisz\'ar-Kullback inequality}
The main tool in establishing results of this paper is the use of  relative entropy defined by
\begin{equation}\label{e:H}
\cH = \int_\T \rho \log \frac{\rho}{\brho} \dx
= \int_\T \rho \log \rho \dx
- M_0 \log \bar{\rho},
\end{equation}
where $\brho = \frac{1}{2\pi} M_0$ and  $M_0 = \int_\T \rho(x) \dx$ is the total mass.  The classical Csisz\'{a}r-Kullback inequality states 
\begin{equation}\label{e:CK}
\|\rho - \bar{\rho}\|_{L^1}^2 \le 4\pi \brho \cH.
\end{equation}
Furthermore, by the elementary inequality $\log x \le x - 1$, we also have 
\begin{equation}
\cH \le \int_\T \rho \left( \frac{\rho}{\brho} - 1 \right) \dx
= \brho^{-1} \left\| \rho - \brho \right\|_{L^2}^2.
\end{equation}
Thus we obtain the two sided bounds
\begin{equation}
\label{e:L1HL2}
\frac{1}{4\pi} \left\| \rho - \brho \right\|_{L^1}^2 \le \brho \cH \le \left\| \rho - \brho \right\|_{L^2}^2.
\end{equation}

\subsection{Evolution of the entropy}
At the heart of the argument is the equation on the entropy \eqref{e:H} which one obtains testing the continuity equation \eqref{e:maind} with $\log \rho + 1$:
\begin{align*}
(\rho \log \rho)_t 
& = \rho_t (\log \rho + 1) 
= -(\rho u)' (\log \rho + 1) \\
& = - \rho' (\log \rho + 1)u - \rho u' (\log \rho + 1) \\
& = - (\rho \log \rho)' u - (\rho \log \rho) u' - \rho u' \\
& = -[u (\rho \log \rho)]' - \rho u'
= -[u (\rho \log \rho)]' - \rho^2 q  + \rho \cL_\psi \rho.
\end{align*}
Therefore, 
\begin{equation}
\label{e:dtentropy}
\frac{\dd\cH}{\dt} = 
\frac{\dd}{\dt} \int_\T \rho \log \rho \dx 
= - \int_\T \rho^2 q \dx -  \int_{\T^2} \psi(x,y)(\rho(x) - \rho(y)) \rho(x)\dx\dy.
\end{equation}
Noting that $\int_\T \rho q \dx = \int_\T e \dx = 0 $, we can subtract $\brho$ from one density in the first integral on the left hand side. After additionally symmetrizing the last integral we obtain 
\begin{equation}
\label{e:dtentropy2}
\frac{\dd\cH}{\dt} = - \int_\T \left(\rho - \brho \right) \rho q \dx - \frac12  \int_{\T^2} \psi(x,y)|\rho(x) - \rho(y)|^2 \dx\dy.
\end{equation}

\subsection{Bounds on the dissipation}

If our kernels $\psi$ were global, it would be easy to get a positive lower bound on the dissipation term:
\begin{equation}
\label{e:globkerdiss}
 \int_{\T^2}\psi(x,y)|\rho(x) - \rho(y)|^2 \dx\dy
\ge (\inf \psi)  \int_{\T^2} |\rho(x)- \rho(y)|^2\dx\dy = 2(\inf \psi)   \|\rho - \brho\|_{L^2}^2.  
\end{equation}
Since in all our cases we have a non-trivial lower bound on the kernel only near the diagonal $\{(x,y)\in \T^2: |x-y|<R_0\}$, we need a substitute for \eqref{e:globkerdiss} stated in the following Lemma. 

\begin{LEMMA}
The following inequality holds:
\label{l:poinc}
\begin{equation}
\label{e:poinc}
\frac12 \int_{\T} \int_{|z|<R_0} |\rho(x) - \rho(x+z)|^2\dz\dx\ge c(R_0)\|\rho - \brho\|_{L^2}^2.
\end{equation}
\end{LEMMA}

\begin{proof}
Let $\chi$ be a nonnegative cutoff function on $\T$ with support in $B_{R_0}(0)$, which is constant on $B_{R_0/2}$ and has integral $1$.  Then $\widehat{\chi}(0) = 1$ and  $|\widehat{\chi}(k)|<1$ for all $k\in \Z\backslash \{0\}$.  On the other hand,  we have by the Riemann-Lebesgue Lemma that $\widehat{\chi}(k)\to 0$ as $k\to \infty$; therefore we have in fact that $|\widehat{\chi}(k)| \le 1 - \e$ for some $\e>0$ depending only on $R_0$ ($k\ne 0$).  Define $\brho_{R_0}(x) = \chi * \rho(x)$, so that 
\[
(\rho-\brho_{R_0})^{\widehat{\,}}(k) = (1 - \widehat{\chi}(k))\widehat{\rho}(k).
\]
Consequently 
\[
|(\rho-\brho_{R_0})^{\widehat{\,}}(k)| \ge \e |\widehat{\rho}(k)|, 
\quad 
k\in \Z, \; k\ne 0,
\]
and $\widehat{\rho}(0) = \widehat{\brho}_{R_0}(0)$.  Thus
\[
\|\rho - \brho\|_{L^2}^2 = \sum_{k\in \Z\backslash\{0\}} |\widehat{\rho}(k)|^2 \le \e^{-2} \sum_{k\in \Z} |(\rho-\brho_{R_0})^{\widehat{\,}}(k)|^2
= \e^{-2}\|\rho - \brho_{R_0}\|_{L^2}^2.
\]
By the fact that $\int_\T \chi = 1$ and Minkowski, we have 
\begin{align*}
\|\rho - \brho_{R_0}\|_{L^2}^2
& = \left\| \int_\T \chi(y)(\rho(\cdot) - \rho(\cdot -y))\dy \right\|_{L^2}^2 
\le \int_{|y|<R_0} \|\rho(\cdot ) - \rho(\cdot -y)\|_{L^2}^2 \dy \\
& = \int_{\T} \int_{|z|<R_0} |\rho(x) - \rho(x+z)|^2\dz\dx.
\end{align*}
Combining the inequalities above yields
\[
\|\rho - \brho\|_{L^2}^2 \le \e^{-2} \int_{\T} \int_{|z|<R_0} |\rho(x) - \rho(x+z)|^2\dz\dx;
\]
taking $c(R_0) = 2\e^{2}$ finishes the proof.
\end{proof}

Estimating the dissipation term now becomes trivial:
\begin{equation}
\label{e:disslwr}
\frac12 \int_{\T^2} \psi(x,y,t)|\rho(x) - \rho(y)|^2\dx\dy \ge  c(R_0) \underline{\psi}(t) \|\rho - \brho\|_{L^2}^2,
  \end{equation}
 where 
\[
\underline{\psi}(t) = \inf\{\psi(x,y,t):|x-y|<R_0\}.
\]
Of course, if $\psi$ is a Lipschitz kernel \eqref{e:Lipker}, then 
\begin{equation}\label{e:psiLip}
	\underline{\psi}(t) \geq \l.
\end{equation}
If $\psi$ is a topological kernel, then estimating $\psi$ from below requires a choice, as the topological distance $\dd(x,y,t)$ can be  bounded above by $|x-y| \|\rho(t)\|_{L^\infty}$, by $M_0$, or by some combination of the two.  In general, we have 
\begin{equation}
\label{e:psilwr}
\psi(x,y,t)\ge \frac{\l \chi_{R_0}(|x-y|)}{|x-y|^{1+\a-\eta} M_0^\eta \|\rho(t)\|_{L^\infty}^{\t-\eta}}, 
\quad \quad 
\eta\in [0,\t], \; t\ge 0.
\end{equation}
If $\t\le \a+1$, then we can put $\eta = \tau$ to eliminate the density amplitude and obtain a lower bound on $\underline{\psi}(t)$ which is obviously uniform in time:
\begin{equation}
\label{e:psits}
\underline{\psi}(t)\ge \l R_0^{-1-\a+\t}M_0^{-\t}, \quad \quad
\text{ for all } t\ge 0, \quad (\tau\le 1+\a).
\end{equation}
If however $\tau > 1+\a$, then the choice $\eta = \tau$ is no longer compatible with a positive lower bound on $\underline{\psi}$, since in this case we have a positive power of $|x-y|$ on the right side of \eqref{e:psilwr}.  In light of this, we will mostly use the following more general lower bound, which is valid for any $\t\ge 0$.
\begin{equation}
\label{e:psilwreta}
\underline{\psi}(t) \ge \frac{\l}{R_0^{1+\a - \eta} M_0^\eta \|\rho(t)\|_{L^\infty}^{\t-\eta}}
\quad \quad 
\eta\in [0,\min\{\t, 1+\a\}],
\end{equation}
Of course, when $\eta \ne \t$, additional bounds on the density amplitude are needed in order to obtain a uniform estimate on $\underline{\psi}(t)$.  We provide such bounds in Section \ref{s:densamp}.  

It turns out that the choice of $\eta\in [0,\min\{1+\a, \t\}]$ in \eqref{e:psilwreta} has very little effect on the result of Theorem \ref{t:topo}.  Eventually, we will set $\eta = \t$ or $\eta = 1+\a$, simply to clean up some of the exponents that appear in our bounds; however, we will carry along the $\eta$-dependence to show that we don't lose any information by making these choices.  The only place where the lower bound \eqref{e:psits} gives an advantage over the more general \eqref{e:psilwreta} is in the proof of Theorem \ref{t:exp}.  In the case $e_0 = 0$, the bound \eqref{e:psilwreta} is uniform in time, with $\|\rho_0\|_{L^\infty}$ replacing $\|\rho(t)\|_{L^\infty}$.  As a result, the constant $c_2$ from Theorem \ref{t:exp} carries a dependence on $\|\rho_0\|_{L^\infty}^{\t-\eta}$. If $\t\le 1+\a$, then we can choose $\eta=\t$ (i.e., use \eqref{e:psits} instead of \eqref{e:psilwreta}) and eliminate this dependency.

Finally, we note that for the purposes of bounding the density amplitude, using $\eta = \t$ in \eqref{e:psilwr} will be crucial, in contrast to the computations on which \eqref{e:psilwreta} depends. Our density bounds will rely on the full kernel $\psi(x,y,t)$ rather than just its lower bound $\underline{\psi}(t)$ near the diagonal.

\subsection{The entropy equation revisited}

We now return to the evolution equation \eqref{e:dtentropy2} and make estimates; in doing so we will prove Theorem \ref{t:exp}, and we will reduce the proofs of Theorems \ref{t:Lip} and \ref{t:topo} to proving the relevant bounds on the density amplitude.  

In what follows, the constant $c=c(R_0)$ depends only on $R_0$ but may change from line to line.  We have
\begin{align*}
\dot{\cH}(t) 
& \le \|\rho(t)\|_{L^\infty} \|q_0\|_{L^\infty} \|\rho(\cdot, t) - \brho\|_{L^1} - c(R_0) \underline{\psi}(t) \| \rho(\cdot, t) - \brho\|_{L^2}^2 \\
& \le \|\rho(t)\|_{L^\infty} \|q_0\|_{L^\infty} \sqrt{4\pi \brho \cH(t)} -   c(R_0)\underline{\psi}(t) \brho \cH(t).
\end{align*}
Setting $Y = \sqrt{\cH}$, we find
\[
\dot{Y}(t) \leq  \|\rho(t)\|_{L^\infty} \|q_0\|_{L^\infty} \sqrt{\pi \brho} - c(R_0) \underline{\psi}(t)\brho Y(t).
\]
Using Gr\"onwall's lemma we obtain
\begin{multline}
\label{e:gronapp}
\hspace{20 mm} 
Y(t) \le Y_0 \exp\left(- c(R_0)\brho\int_0^t \underline{\psi}(s)\ds\right) \\ + \sqrt{\pi \brho} \|q_0\|_{L^\infty}\int_0^t \|\rho(s)\|_{L^\infty} \exp\left(\!-c(R_0)\brho \int_s^t \underline{\psi}(\tau)\dtau\right)\ds.
\hspace{20 mm} 
\end{multline}
From here, it is easy to prove Theorem \ref{t:exp}.  Indeed, if $e_0 \equiv 0$, then the second term in \eqref{e:gronapp} drops out completely.  Furthermore, $\rho$ satisfies a maximum principle in this case: $\|\rho(t)\|_{L^\infty} \le \|\rho_0\|_{L^\infty}$, so $\underline{\psi}$ is uniformly bounded below by a constant in all cases: $\underline{\psi}(s)\ge \underline{\psi}$.    
\begin{align*}
\|\rho(\cdot, t) - \brho\|_{L^1} 
& \le \sqrt{4\pi \brho}Y(t) \le \sqrt{4\pi \brho \cH_0} \exp( -c(R_0) \brho t \underline{\psi} ) \\ 
& \le \sqrt{4\pi} \|\rho_0 - \brho\|_{L^2} \exp( -\l c_2 t)
\le c\| \rho_0 \|_{L^2} \exp( -\l c_2 t ), 
\end{align*}
where $c_2 = c(R_0) \brho (\underline{\psi}/\l)$.  Thus, in the Lipschitz case, $c_2$ depends only on $R_0$ and $M_0$. In the topological case, $c_2$ depends only on $R_0$, $M_0$, $\a$, and $\t$ when $\t\le 1 + \a$; if $\t>1+\a$, then $c_2$ depends additionally on $\|\rho_0\|_{L^\infty}$. This completes the proof of Theorem \ref{t:exp}.

If $\psi$ is a Lipschitz kernel or a topological kernel with $\t\le 1+\a$, then we know that $\underline{\psi}(t)$ is uniformly bounded below, hence the first term on the right side of \eqref{e:gronapp} vanishes as $t\to \infty$. In fact, the same is true for topological kernels with $\t>1+\a$; this will be apparent once we prove upper bounds on the density amplitude.  We write
\begin{equation}
\label{e:limsup1}
\limsup_{t\to \infty} \|\rho(\cdot, t) - \brho\|_{L^1} \le M_0 \|q_0\|_{L^\infty} \limsup_{t\to \infty} \int_0^t \|\rho(s)\|_{L^\infty} \exp\left( -c(R_0)\brho \int_s^t \underline{\psi}(\tau)\dtau \right) \ds.
\end{equation}
If $\psi$ is a Lipschitz kernel, then this estimate becomes simply
\begin{equation}\label{e:dest}
\limsup_{t \to \infty} \| \rho(\cdot,t) - \bar{\rho}\|_{L^1}  \leq \frac{ \|q_0\|_{L^\infty}}{ \l  c(R_0) } \limsup_{t\to \infty} \|\rho(t) \|_{L^\infty}.
\end{equation}
Thus the proof of Theorem \ref{t:Lip} is reduced to estimating the density amplitude.  The situation is similar for topological kernels.  Upon substituting \eqref{e:psilwreta} into \eqref{e:limsup1}, the usefulness of the following Lemma becomes apparent:

\begin{lemma}
Suppose $f$ is a positive bounded function on $\R^+$ with  $\limsup_{t\to \infty} f(t) = L$. Let $c>0$, $\d\geq 0$. Then
\[
\limsup_{t\to \infty} \int_0^t f(s) \exp\left\{ - \int_s^t \frac{c}{f^\d(l)} \dl  \right\} \ds \leq \frac{L^{1+\d}}{c}.
\]
\end{lemma}
\begin{proof}
Fix $\e>0$, and $T_0$ so that $f(t) < L+\e$ for $t>T_0$. Since $f$ is bounded, it is clear that 
\[
\int_0^{T_0} f(s) \exp\left\{ - \int_s^t \frac{c}{f^\d(l)} \dl  \right\} \ds \to 0,
\]
as $t \to \infty$. Let us estimate the rest
\[
\int_{T_0}^t f(s) \exp\left\{ - \int_s^t \frac{c}{f^\d(l)} \dl  \right\} \ds  \leq \int_{T_0}^t (L+\e)  \exp\left\{ - (t-s) \frac{c}{(L+\e)^\d}   \right\} \ds \leq \frac{(L+\e)^{1+\d}}{c}.
\]
This finishes the proof. 
\end{proof}

Applying the Lemma to \eqref{e:limsup1} and \eqref{e:psilwreta}, we conclude that for any $\eta\in [0,\min\{\t, 1+\a\}]$, we have
\begin{equation}
\label{e:disord}
\limsup_{t\to \infty} \|\rho(\cdot, t) - \brho\|_{L^1} 
\le 
\frac{M_0^{\eta} R_0^{1 + \a - \eta} \|q_0\|_{L^\infty}}{\l c(R_0)} \limsup_{t\to \infty} \|\rho(t)\|_{L^\infty}^{1+\t-\eta}.
\end{equation}
Thus in all cases, all that remains is a large-time bound on the density amplitude.

\subsection{Bounds on the Density Amplitude}

\label{s:densamp}

Throughout our discussion of bounds on the density amplitude, we will make use of the following differential inequality: If f $\dot{X}(t) \le AX(t)[B-X(t)]$, where $A$ and $B$ are positive constants and $X(t)$ is a positive function, then 
\begin{equation}\label{e:X}
X(t) \le \frac{B X(0)}{X(0) + (B-X(0))\exp(-ABt)}.
\end{equation}
In particular, $\limsup_{t\to \infty} X(t) \leq B$.

\subsubsection{Case of Lipschitz kernels} \label{s:Lip} 

Let $\rho_+(t)$ denote the maximum value of $\rho$ at time $t$, and let $x_+$ denote the $x$-value where the maximum is achieved. Then if $ \|q_0\|_{L^\infty} < \|\psi\|_{L^1}$, one can get an upper bound on $\|\rho(t)\|_{L^\infty}$ by integrating the differential inequality derived below. 
\begin{align*}
	\frac{\dd}{\dt} \rho_+(t) & = -\rho_+(t) u'(x_+,t) = -\rho_+(t)^2 q(x_+,t) + \rho_+(t)\int_\T \psi(x_+ - y)(\rho(y,t)-\rho_+(t))\dy  \\
	& \le (\|q_0\|_{L^\infty} - \|\psi\|_{L^1} )\rho_+(t)^2 + \|\psi\|_{L^\infty} M_0 \rho_+(t) \\
	& =   ( \|\psi\|_{L^1} - \|q_0\|_{L^\infty} )\rho_+(t) \left[ \frac{\|\psi\|_{L^\infty} M_0}{\|\psi\|_{L^1} - \|q_0\|_{L^\infty} } - \rho_+(t) \right].
\end{align*}
In view of \eqref{e:X} we obtain
\[
\limsup_{t\to \infty} \|\rho(t) \|_{L^\infty} \le \frac{\|\psi\|_{L^\infty} M_0}{\|\psi\|_{L^1} - \|q_0\|_{L^\infty} }.
\]
Plugging into \eqref{e:dest} we conclude
\[
\limsup_{t \to \infty} \| \rho(\cdot,t) - \bar{\rho}\|_{L^1}  \leq \frac{ M_0  \|q_0\|_{L^\infty} \|\psi\|_{L^\infty}}{ \l  c(R_0) (\|\psi\|_{L^1} - \|q_0\|_{L^\infty}) }.
\]

\subsubsection{Case of topological kernels with $0\leq \t \leq \a$}

The case of topological kernels follows the same strategy as for Lipschitz kernels, with additional technicalities.  Note that in order for the dissipation to compete with the quadratic term $q\rho^2$, we must choose $\eta = \tau$ in \eqref{e:psilwr}; otherwise the associated power of $\rho_+$ will be less than $2$.

For any $r\in (0,R_0)$, we have
\begin{align*}
\frac{\dd}{\dt} \rho_+(t) 
& = -q(x_+,t)\rho_+(t)^2 + \rho_+(t) \int \psi(x_+, x_+ + z)(\rho(x_+ + z,t) - \rho_+(t))\dz \\
& \le \|q_0\|_{L^\infty} \rho_+(t)^2 + \l M_0^{-\t} \rho_+(t) \int_{r<|z|<R_0} \frac{\rho(x_+ + z,t) - \rho_+(t)}{|z|^{1+\a-\t}}\dz \\
& \le \left[\|q_0\|_{L^\infty} - \l M_0^{-\t} \int_r^{R_0} s^{-1-\a+\t}\ds \right] \rho_+(t)^2 +  \l r^{-1-\a+\t} M_0^{1-\t} \rho_+(t) \\
& = [\l M_0^{-\t} I(r) - \|q_0\|_{L^\infty}] \rho_+(t) \left[ \frac{\l r^{-1-\a+\t}M_0}{\l I(r) - M_0^\t \|q_0\|_{L^\infty}} -\rho_+(t) \right], 
\end{align*}
where in the last line we have written $I(r)$ for the integral $\int_r^{R_0} s^{-1-\a+\t}\ds$.  Thus,
\[
 \limsup_{t\to \infty} \|\rho(t) \|_{L^\infty} \le \frac{\l r^{-1-\a+\t}M_0}{\l I(r) - M_0^\t \|q_0\|_{L^\infty}}.
\]
We choose $r>0$ that minimizes the right hand side, hence,
\begin{equation}
\label{e:chooser}
\l I(r) - M_0^\t \|q_0\|_{L^\infty} = \frac{\l r^{\t-\a}}{1+\a-\t},
\end{equation}
and consequently 
\[
\limsup_{t\to \infty} \|\rho(t) \|_{L^\infty} \le \frac{r^{-1-\a+\t}M_0}{\l I(r) - M_0^\t \|q_0\|_{L^\infty}} = (1+\a-\t)M_0 r^{-1}.
\]
Note that the choice \eqref{e:chooser} amounts to putting 
\begin{equation}
\label{e:defr}
r = \left\{ 
\begin{array}{lcl}
\left((1+\a-\t)[(\a-\t) \l^{-1} M_0^\t \|q_0\|_{L^\infty} + R_0^{\t-\a}]\right)^{-\frac{1}{\a-\t}},
& & \t<\a \\ \\
R_0 \exp(-1 - \l^{-1} M_0^\t \|q_0\|_{L^\infty}),
& & \t = \a.
\end{array}
\right.
\end{equation}
The large time upper bound on $\rho$ then becomes
\[
\limsup_{t\to \infty} \|\rho(t) \|_{L^\infty}
\le 
\left\{
\begin{array}{lcl}
(1+\a-\t)^{1+\frac{1}{\a-\t}} M_0 [(\a - \t) \l^{-1} M_0^\t \|q_0\|_{L^\infty} + R_0^{\t-\a}]^{\frac{1}{\a-\t}}, & & \t<\a \\
\\
M_0 R_0 \exp\left( 1 + \l^{-1} M_0^\t \|q_0\|_{L^\infty} \right),
& & \t = \a. \\
\end{array}\right.
\]
Plugging into \eqref{e:disord} we obtain
\begin{equation}\label{e:disordr}
\limsup_{t \to \infty} \| \rho(\cdot,t) - \bar{\rho}\|_{L^1}  \leq
\left\{
\begin{array}{lcl}
 \frac{M_0^{1+\t}  \|q_0\|_{L^\infty}}{ \l  R_0^\eta c(R_0,\a,\t)  }  [(\a - \t) \l^{-1} M_0^\t \|q_0\|_{L^\infty} + R_0^{\t-\a}]^{\frac{1+\t - \eta}{\a-\t}}, & & \t<\a \\
\\
\frac{M_0^{1+\t}  \|q_0\|_{L^\infty}}{ \l  R_0^{2\eta} c(R_0,\a,\t)  }  \exp\left[ (1+\t - \eta)( 1 + \l^{-1} M_0^\t \|q_0\|_{L^\infty}) \right],
& & \t = \a. \\
\end{array}\right.
\end{equation}
We choose $\eta = \t$ to obtain the bound in Theorem \ref{t:topo}.

\subsubsection{Case of topological kernels with $\a<\t<\a+1$}

In the case above, where $\t\le \a$, a key step in the argument was to use the fact that $I(r)\to +\infty$ as $r\to 0^+$.  This is what allowed us to bound $\rho$ in terms of $\|q_0\|_{L^\infty}$, regardless of the size of $q_0$.  If $\a<\t<\a+1$, then $I(0)$ is finite; therefore, we need a smallness condition on $\|q_0\|_{L^\infty}$ in order for the above argument to give an upper bound on the density, similar to the case of Lipschitz kernels.  In particular, we need 
\[
\l^{-1} M_0^\t \|q_0\|_{L^\infty} < I(0) = R_0^{\t-\a}/(\t-\a).
\]
 If this is the case, then it follows that there exists $r_0>0$ such that $I(r_0) = \l^{-1} M_0^\t \|q_0\|_{L^\infty}$ and $I(r)> \l^{-1} M_0^\t \|q_0\|_{L^\infty}$ for $r\in (0,r_0)$.  The value of this $r_0$ is given by 
\[
r_0^{\t-\a} = R_0^{\t-\a}-(\t-\a)\l^{-1}  M_0^\t \|q_0\|_{L^\infty}.
\]
If 
\[
r^{\t-\a} = (1+\a-\t)[R_0^{\t-\a}-(\t-\a) \l^{-1} M_0^\t\|q_0\|_{L^\infty}]=(1+\a-\t) r_0^{\t-\a},
\]
then clearly $r\in (0, r_0)$.  Furthermore, this formula agrees with that given in \eqref{e:defr} (for $\t<\a$), so by the same manipulations as before, we have 
\begin{align*}
\limsup_{t\to \infty} \|\rho(t) \|_{L^\infty}
& \le \frac{ r^{-1-\a+\t}M_0}{I(r) - \l^{-1} M_0^\t \|q_0\|_{L^\infty}} = (1+\a-\t)M_0 r^{-1} \\
& = (1+\a-\t)^{1-\frac{1}{\t-\a}} M_0 [R_0^{\t-\a}-(\t - \a)\l^{-1}  M_0^\t \|q_0\|_{L^\infty}]^{\frac{1}{\a-\t}}
\end{align*}
Plugging this into \eqref{e:disord} we obtain
\[
\limsup_{t \to \infty} \| \rho(\cdot,t) - \bar{\rho}\|_{L^1}  \leq   \frac{M_0^{1+\t}  \|q_0\|_{L^\infty}}{ \l R_0^\eta c(R_0,\a,\t)  [(\a - \t) \l^{-1} M_0^\t \|q_0\|_{L^\infty} + R_0^{\t-\a}]^{\frac{1+\t - \eta}{\t-\a}}  } .
\]
Once again, we choose $\eta = \t$ to obtain the bound in Theorem \ref{t:topo}.

\subsubsection{Case of topological kernels with $\t\ge \a+1$}

When $\t\ge \a+1$, we still require 
\begin{equation}\label{e:qsmall4}
	\l^{-1} M_0^\t \|q_0\|_{L^\infty} < I(0) = R_0^{\t-\a}/(\t-\a)
\end{equation} 
in order to get an upper bound on the density. However, our initial estimate on the time derivative of $\rho_+(t)$ needs minor adjustments, since the power $1+\a - \t$ associated to the geometric part of the kernel is no longer positive:
\begin{align*}
\frac{\dd}{\dt} \rho_+(t) 
& \le \|q_0\|_{L^\infty} \rho_+(t)^2 + \l M_0^{-\t} \rho_+(t) \int_{|z|<R_0} \frac{\rho(x_+ + z,t) - \rho_+(t)}{|z|^{1+\a-\t}}\dz \\
& \le \left[\|q_0\|_{L^\infty} - \l M_0^{-\t} I(0) \right] \rho_+(t)^2 + \l R_0^{-1-\a+\t} M_0^{1-\t} \rho_+(t) \\
& = [\l M_0^{-\t} I(0) - \|q_0\|_{L^\infty}] \rho_+(t) \left[ \frac{\l R_0^{\t-1-\a}M_0}{\l I(0) - M_0^\t \|q_0\|_{L^\infty}} -\rho_+(t) \right].
\end{align*}
Our long-time bound becomes
\begin{equation}\label{e:densupper}
\limsup_{t\to \infty} \|\rho(t) \|_{L^\infty} \le  \frac{\l R_0^{\t-1-\a}M_0}{\l I(0) - M_0^\t \|q_0\|_{L^\infty}} 
= \frac{(\t-\a) R_0^{\t-1-\a}M_0}{R_0^{\t-\a} - (\t-\a) \l^{-1} M_0^\t \|q_0\|_{L^\infty}}.
\end{equation}
Plugging into \eqref{e:disord} yields
\[
\limsup_{t \to \infty} \| \rho(\cdot,t) - \bar{\rho}\|_{L^1} \leq \frac{M_0^{1+\t} \|q_0\|_{L^\infty}}{\l R_0^{\eta(\t - 1 - \a)}  c(R_0, \a, \t) [  R_0^{\t-\a} - (\t-\a) \l^{-1} M_0^\t \|q_0\|_{L^\infty}]^{1+\tau-\eta}},  
\]
which gives the inequality of Theorem \ref{t:topo} when $\eta = 1+\a$.


\end{document}